\documentclass[11pt,a4paper,dvipsnames,x11names]{amsart}
\usepackage[margin=30truemm]{geometry}
\usepackage{fancyhdr,appendix,graphicx,layout}
\usepackage[T1]{fontenc}
\usepackage{amsmath,amsthm,amssymb,tensor}
\usepackage{lmodern} % Standard Computer Modern-style text and mathematics
\usepackage{microtype}
\usepackage{enumitem}
\usepackage{physics}
\usepackage{adjustbox}
\usepackage[numbers]{natbib}
\usepackage{tikz}
\usepackage{tikz-cd}
\usetikzlibrary{positioning,arrows.meta,decorations.markings,calc}
\tikzset{cells={font=\everymath\expandafter{\the\everymath\displaystyle}},}
\usepackage{caption}
\usepackage[pagewise]{lineno}
\usepackage{float}
\renewcommand{\emph}[1]{\textcolor{blue}{\textit{#1}}}
\usepackage{hyperref}
\usepackage{aliascnt}
\hypersetup{colorlinks=true,linkcolor=red,filecolor=blue,urlcolor=orange,citecolor=orange}

\makeatletter
\newcommand\ReDeclareMathOperator[2]{\begingroup \escapechar\m@ne\xdef\@gtempa{{\string#1}}\endgroup\expandafter\@ifundefined\@gtempa{\@latex@error{Command \string#1 undefined}\@ehc}\relax\let\@ifdefinable\@rc@ifdefinable\DeclareMathOperator#1{#2}}
\makeatother
\theoremstyle{definition}
\newtheorem{thm}{Theorem}[section]
\newtheorem{main}{Theorem}
\newaliascnt{dfn}{thm}
\newtheorem{dfn}[dfn]{Definition}
\aliascntresetthe{dfn}
\newaliascnt{dfnprop}{thm}

\aliascntresetthe{dfnprop}
\newaliascnt{dfnlem}{thm}

\aliascntresetthe{dfnlem}
\newaliascnt{prop}{thm}
\newtheorem{prop}[prop]{Proposition}
\aliascntresetthe{prop}
\newaliascnt{lem}{thm}
\newtheorem{lem}[lem]{Lemma}
\aliascntresetthe{lem}
\newaliascnt{cor}{thm}
\newtheorem{cor}[cor]{Corollary}
\aliascntresetthe{cor}
\newaliascnt{clm}{thm}

\aliascntresetthe{clm}
\newaliascnt{ass}{thm}

\aliascntresetthe{ass}
\newaliascnt{cond}{thm}

\aliascntresetthe{cond}
\newaliascnt{nota}{thm}
\newtheorem{nota}[nota]{Notation}
\aliascntresetthe{nota}
\newaliascnt{conj}{thm}

\aliascntresetthe{conj}
\newaliascnt{fct}{thm}

\aliascntresetthe{fct}
\newaliascnt{rmk}{thm}
\newtheorem{rmk}[rmk]{Remark}
\aliascntresetthe{rmk}
\newaliascnt{eg}{thm}
\newtheorem{eg}[eg]{Example}
\aliascntresetthe{eg}
\numberwithin{equation}{section}
\makeatletter\let\c@equation\c@thm\makeatother

\usepackage{cleveref}
\crefname{thm}{Theorem}{Theorems}
\crefname{dfn}{Definition}{Definitions}
\crefname{dfnprop}{Definition-Proposition}{Definition-Propositions}
\crefname{dfnlem}{Definition-Lemma}{Definition-Lemmas}
\crefname{prop}{Proposition}{Propositions}
\crefname{lem}{Lemma}{Lemmas}
\crefname{cor}{Corollary}{Corollaries}
\crefname{clm}{Claim}{Claims}
\crefname{ass}{Assumption}{Assumptions}
\crefname{cond}{Condition}{Conditions}
\crefname{nota}{Notation}{Notations}
\crefname{conj}{Conjecture}{Conjectures}
\crefname{fct}{Fact}{Facts}
\crefname{rmk}{Remark}{Remarks}
\crefname{eg}{Example}{Examples}
\crefname{figure}{Figure}{Figures}
\crefname{table}{Table}{Tables}
\crefname{section}{Section}{Sections}
\crefname{subsection}{Subsection}{Subsections}
\crefname{appendix}{Appendix}{Appendices}
\crefname{equation}{}{}
\crefname{main}{Theorem}{Theorems}

\newcommand{\D}{\mathcal{D}}

\renewcommand{\H}{\mathcal H}

\newcommand{\K}{\mathcal K}
\renewcommand{\L}{\mathcal L}

\newcommand{\m}{\mathfrak m}
\newcommand{\n}{\mathfrak n}

\renewcommand{\S}{\mathcal S}
\newcommand{\T}{\mathcal T}
\newcommand{\U}{\mathcal U}
\newcommand{\V}{\mathcal V}

\newcommand{\RR}{\mathbf{R}}
\newcommand{\LL}{\mathbf{L}}
\newcommand{\FF}{\mathbb{F}}

\newcommand{\ZZ}{\mathbb{Z}}

\DeclareMathOperator{\Filt}{Filt}
\DeclareMathOperator{\add}{add}

\ReDeclareMathOperator{\top}{top}
\DeclareMathOperator{\simp}{Sim}

\DeclareMathOperator{\proj}{proj}

\DeclareMathOperator{\colim}{colim}

\DeclareMathOperator{\Mod}{Mod}
\let\mod\relax\DeclareMathOperator{\mod}{mod}
\DeclareMathOperator{\fl}{fl}
\DeclareMathOperator{\Hom}{Hom}
\DeclareMathOperator{\HHom}{\mathcal Hom}
\DeclareMathOperator{\End}{End}

\DeclareMathOperator{\RHom}{\RR Hom}

\DeclareMathOperator{\REnd}{\RR End}

\DeclareMathOperator*{\ten}{\otimes}

\DeclareMathOperator{\Tor}{Tor}
\DeclareMathOperator{\thick}{thick}

\DeclareMathOperator{\silt}{silt}
\DeclareMathOperator{\twosilt}{2\text{-}silt}
\DeclareMathOperator{\SMC}{SMC}
\DeclareMathOperator*{\holim}{holim}
\DeclareMathOperator*{\hocolim}{hocolim}

\DeclareMathOperator{\cocone}{cocone}
\DeclareMathOperator{\per}{per}

\DeclareMathOperator{\Dbfl}{\mathcal D^b_{fl}}

\title{Silting Theory and Derived Base Change}
\author{Riku Fushimi}
\date{\today}

\begin{document}

\begin{abstract}
We show that the silting theory of a locally finitely generated non-positive dg algebra over a complete local noetherian dg ring is determined by its derived closed fiber. More generally, we establish the invariance of silting theory under suitable derived base change. The key ingredient is a silting--simple-minded correspondence between the perfect derived category and the derived category with finite-length cohomology.
\end{abstract}
\maketitle

\setcounter{tocdepth}{1}
\tableofcontents

%%%%%%%%%%%%%%%%%%%%%%%%%%%%%%%%%%%%%%%%%%%%%%%%%%%%%%%%%%%%%%%%%%%%%%%%%%
\section{Introduction}
Tilting objects are a basic source of derived equivalences, while silting objects provide a flexible generalization admitting mutation \cite{AI12}. They also control structural data such as $t$-structures and co-$t$-structures. For a finite-dimensional algebra $\Lambda$, Koenig and Yang \cite{KoY14} established a bijection between basic silting objects in $\K^b(\proj\Lambda)$ and simple-minded collections in $\D^b(\mod\Lambda)$, extending earlier work of Keller--Vossieck \cite{KV88}.

The aim of this paper is to develop this correspondence over a complete local noetherian dg base and to apply it to derived base change. A commutative non-positive dg ring $R$ is called \emph{complete local noetherian} if $H^0(R)$ is a complete local noetherian ring and $H^i(R)$ is finitely generated over $H^0(R)$ for every $i<0$. For example, every complete local noetherian ring, viewed as a dg ring concentrated in degree zero, satisfies these conditions. We write $(H^0(R),\m,k)$ for its degree-zero local ring and residue field. Let $A$ be a non-positive dg $R$-algebra such that every $H^i(A)$ is finitely generated over $H^0(R)$.

Our first main result extends the silting--simple-minded correspondence to this setting.

\begin{main}(=\cref{thm:ST})\label{main:ST}
	Let $R$ be a complete local noetherian non-positive commutative dg ring and let $A$ be a locally finitely generated non-positive dg $R$-algebra. Then there is a bijection
	\[
		\phi\colon\silt(\per(A))\xrightarrow{\sim}\SMC(\Dbfl(A)).
	\]
\end{main}

The proof uses Koszul duality with respect to the standard semisimple object $S_A$. For Koszul-duality approaches to the silting--simple-minded correspondence for non-positive dg algebras, see Su--Yang \cite{SY19} and Fushimi \cite{F23}. The base is allowed to be a complete local noetherian dg ring, and $H^*(A)$ need only be degreewise finitely generated over $H^0(R)$, with no lower bound on its cohomology. Completeness enters in the double-centralizer argument of \cref{prop:double-centralizer}.

Our second main result shows that silting theory is determined by the derived closed fiber $A_k:=A\ten_R^{\LL}k$. While base change preserves silting objects, lifting them is less straightforward. We resolve this problem through the silting--simple-minded correspondence of \cref{main:ST}.

\begin{main}(=\cref{thm:bij-fib})\label{main:bij-fib}
	Derived extension of scalars induces a poset isomorphism
	\[
		-\ten_R^{\LL}k:
		(\silt(\per(A)),\ge)
		\xrightarrow{\sim}
		(\silt(\per(A_k)),\ge).
	\]
	It fits into the commutative diagram
	\[
		\begin{tikzcd}
			\silt(\per(A)) \ar[r,"-\ten_R^{\LL}k"] \ar[d,"\phi"',"\sim"sloped] & \silt(\per(A_k)) \ar[d,"\phi","\sim"'sloped] \\
			\SMC(\Dbfl(A)) & \SMC(\Dbfl(A_k)) \ar[l,"\iota"'].
		\end{tikzcd}
	\]
	Here $\iota$ denotes restriction of scalars.
\end{main}

Finally, let $R\to S$ be a local morphism of complete local noetherian non-positive commutative dg rings inducing an isomorphism on residue fields, which we identify with $k$. Since
\[
	(A\ten_R^{\LL}S)\ten_S^{\LL}k\simeq A\ten_R^{\LL}k,
\]
the closed-fiber theorem gives the following base-change invariance.

\begin{main}(=\cref{thm:R-S})\label{main:R-S}
	Derived extension of scalars induces a poset isomorphism
	\[
		(\silt(\per(A)),\ge)
		\xrightarrow{\sim}
		\bigl(\silt(\per(A\ten_R^{\LL}S)),\ge\bigr).
	\]
\end{main}

For ordinary complete local noetherian rings, \cref{main:bij-fib} recovers Kimura's two-term reduction theorem \cite{K} and is in the spirit of the central reduction theorem of Eisele--Janssens--Raedschelders \cite{EJR18}; it also overlaps with the Tor-independent quotient results of Gnedin \cite{G}. See \cref{section:relation}. In contrast with those results, the present formulation retains the derived tensor product, imposes no Tor-independence condition, applies to silting complexes of arbitrary amplitude, and allows the base itself to have negative cohomology. The non-flat family in \cref{eg:triangular-family} illustrates how derived reduction retains a negative-degree arrow invisible in the ordinary fiber. In \cref{eg:ordinary-vs-derived} we show that the ordinary fiber can be silting-discrete while the derived fiber is silting-indiscrete. In \cref{eg:dg-gentle} we give an example over a genuinely dg base: the base is the negatively graded polynomial dg ring $R=k[t]$ with $|t|=-2$, and the algebra is a graded path algebra on a $2$-cycle, where $t$ acts as the sum of the two length-two cycles. 

\medskip
\noindent
{\bf Organization.}

In \cref{section:Pre}, we collect the required facts about dg rings, derived categories with finite-length cohomology, and silting theory. In \cref{section:ST}, we prove \cref{main:ST} by Koszul duality and derived completion. In \cref{section:bij}, we prove \cref{main:bij-fib} and \cref{main:R-S}. In \cref{section:relation}, we compare the ordinary-ring specialization with previous results.

\medskip
\noindent
{\bf Conventions and notation.}

Throughout, $R$ denotes a complete local noetherian non-positive commutative dg ring, and $(H^0(R),\m,k)$ denotes its degree-zero local ring and residue field. Unless otherwise stated, dg modules are right dg modules. Finite length always means finite length over $H^0(R)$. All subcategories are assumed to be full.

For a triangulated category $\T$ and subsets $\U,\V\subseteq\T$, we write $\add\U$, $\thick\U$, and $\Filt\U$ for the additive, thick, and extension closures, respectively, and
\[
	\U\ast\V:=\{X\in\T\mid\text{there is a triangle }U\to X\to V\to\text{ with }U\in\U,\ V\in\V\}.
\]

\medskip
\noindent
{\bf Acknowledgements.}

The author would like to express his sincere gratitude to his supervisor Akira Ishii for his continuous support and encouragement. The author is grateful to Ryu Tomonaga for carefully reading an earlier draft and for many valuable comments and suggestions. The author also thanks Aaron Chan for reading the first version and for pointing out several typos and suggesting improvements to the exposition. The author thanks Sebastian Opper for suggesting the extension of the base-change results to non-positive dg base rings.

\medskip
\noindent
{\bf Use of AI.}

In preparing this paper, the author used OpenAI's GPT-5.6 Sol to assist in generalizing an earlier version from ordinary base rings to commutative dg rings, including revisions to the statements and proofs. The author has reviewed the resulting manuscript and takes full responsibility for its mathematical content.

%%%%%%%%%%%%%%%%%%%%%%%%%%%%%%%%%%%%%%%%%%%%%%%%%%%%%%%%%%%%%%%%%%%%%%%%
\section{Preliminaries}\label{section:Pre}

%%%%%%%%%%%%%%%%%%%%%%%%%%%%%%%%%%%%%%%%%%%%%%%%%%%%%%%%%%%%%%%%%%%%%%%%%%%%%%%%%%%%%%%%%
\subsection{Triangulated categories}

In this subsection, we recall some definitions and facts about triangulated categories. Throughout the subsection, $\T$ denotes a triangulated category with suspension $\Sigma$.

\begin{dfn}\cite{KV88,AI12}
	An object $M\in\T$ is called a \emph{silting object} if
	\begin{itemize}
		\item $\Hom_\T(M,\Sigma^{>0}M)=0$,
		\item $\thick M=\T$.
	\end{itemize}
	Two silting objects $M$ and $N$ are said to be \emph{equivalent} if $\add M=\add N$. We denote by $\silt(\T)$ the set of equivalence classes of silting objects in $\T$.
\end{dfn}

For silting objects $M,N$ in a triangulated category, we use the standard partial order
\[
	M\ge N
	\quad\Longleftrightarrow\quad
	\Hom(M,\Sigma^{>0}N)=0.
\]
This depends only on the equivalence classes of $M$ and $N$.

\begin{dfn}\label{dfn:SMC}\cite{Al09}
	A set $\L=\{L_1,\ldots,L_n\}$ of objects of $\T$ is called a \emph{simple-minded collection} if 
	\begin{itemize}
		\item $\Hom_\T(\L,\Sigma^{<0}\L)=0$,
		\item $\L$ is a semibrick, that is, for every $i,j$,
		\begin{align*}
			\Hom_\T(L_i,L_j)=\begin{cases} \text{a division ring} & \text{if } i=j,\\
			0 & \text{if } i\neq j,\end{cases}
		\end{align*}
		\item $\thick\L=\T$.
	\end{itemize}
	We denote by $\SMC(\T)$ the set of isomorphism classes of simple-minded collections in $\T$.
\end{dfn}

%%%%%%%%%%%%%%%%%%%%%%%%%%%%%%%%%%%%%%%%%%%%%%%%%%%%%%
\subsection{Dg rings and dg algebras}

For basic terminology on dg algebras, dg modules, and derived categories, we refer to \cite{Ke}. By a commutative dg ring we mean a graded-commutative one: $ab=(-1)^{|a||b|}ba$ for homogeneous elements, with the graded Leibniz rule for its differential. A dg ring or dg algebra $C$ is called \emph{non-positive} if $H^{>0}(C)=0$. Its canonical truncation is the dg subalgebra
\[
	(\tau_{\le0}C)^i=
	\begin{cases}
		C^i & i<0,\\
		\ker(d\colon C^0\to C^1) & i=0,\\
		0 & i>0.
	\end{cases}
\]
The inclusion $\tau_{\le0}C\hookrightarrow C$ is a quasi-isomorphism whenever $C$ is non-positive. If $A$ is a dg $R$-algebra, then $\tau_{\le0}A$ is naturally a dg $\tau_{\le0}R$-algebra. When dealing with non-positive dg algebras, we replace the base rings, algebras, and their morphisms simultaneously by their canonical truncations, without further mention.

\begin{dfn}\label{dfn:base}
	A commutative non-positive dg ring $R$ is called \emph{complete local noetherian} if
	\begin{itemize}
		\item $H^0(R)$ is a complete local noetherian ring, and
		\item $H^i(R)$ is a finitely generated $H^0(R)$-module for every $i<0$.
	\end{itemize}
	We write $(H^0(R),\m,k)$ for the degree-zero local ring and its residue field.
\end{dfn}
A morphism $R\to S$ of complete local noetherian non-positive commutative dg rings is called \emph{local} if the induced ring homomorphism $H^0(R)\to H^0(S)$ is local.

The canonical morphism $R\to H^0(R)\to k$ makes $k$ a dg $R$-algebra concentrated in degree zero.

\begin{dfn}\label{dfn:dg}
	Let $A$ be a dg $R$-algebra.
	\begin{itemize}
		\item $A$ is \emph{locally finitely generated} if $H^i(A)$ is a finitely generated $H^0(R)$-module for every $i\in\ZZ$.
		\item $A$ is \emph{locally of finite length} if $H^i(A)$ has finite length as an $H^0(R)$-module for every $i\in\ZZ$.
		\item $A$ is \emph{positive} if $H^{<0}(A)=0$ and $H^0(A)$ is a semisimple ring.
	\end{itemize}
\end{dfn}

\begin{dfn}
	Let $A$ be a dg $R$-algebra.
	\begin{itemize}
		\item $\per(A):=\thick_{\D(A)}A_A$ is the \emph{perfect derived category} of $A$.
		\item
		\[
			\Dbfl(A):=\left\{X\in\D(A)\ \middle|\ H^*(X):=\bigoplus_{i\in\ZZ}H^i(X)\text{ has finite length over }H^0(R)\right\}
		\]
		is the \emph{finite-length derived category} of $A$.
	\end{itemize}
\end{dfn}

For a locally finitely generated non-positive dg $R$-algebra $A$, let
\[
	\D_{\mathrm{fg}}(A):=\{X\in\D(A)\mid H^i(X)\text{ is finitely generated over }H^0(R)\text{ for all }i\in\ZZ\}.
\]
We write $\D_{\mathrm{fg}}(A)^{\le n}:=\D_{\mathrm{fg}}(A)\cap\D(A)^{\le n}$.

If $A$ is non-positive, then the standard $t$-structure on $\D(A)$ has heart $\Mod H^0(A)$. We write $\D(A)^I$ for the full subcategory consisting of objects whose cohomology is concentrated in a subset $I\subseteq\ZZ$, and $\sigma^{\le n},\sigma^{\ge n}$ for the standard truncation functors. In particular, $\D(A)^0$ denotes the heart.

For a length abelian category $\H$, let $\simp(\H)$ denote a set of representatives of the isomorphism classes of its simple objects.

\begin{nota}\label{nota:S_A}
	Let $A$ be a locally finitely generated non-positive dg $R$-algebra. We denote by
	\[
		\S_A:=\simp(\Dbfl(A)^0)=\simp(\fl H^0(A))
	\]
	the standard simple-minded collection in $\Dbfl(A)$, and by $S_A$ the $A$-module corresponding to $\top H^0(A)$. Thus $\add S_A=\add\S_A$ and $\thick S_A=\Dbfl(A)$.
\end{nota}

\begin{lem}\label{lem:Hom-length}
	Let $A$ be a locally finitely generated non-positive dg $R$-algebra. For every $n\ge0$,
	\begin{equation}\label{eq:fg-approx}
		\D_{\mathrm{fg}}(A)^{\le0}
		=\add A\ast\Sigma\add A\ast\cdots\ast\Sigma^n\add A
		\ast\D_{\mathrm{fg}}(A)^{\le-n-1}.
	\end{equation}
	In particular, $\Dbfl(A)$ is Hom-finite-length: for every $X,Y\in\Dbfl(A)$ and every $i\in\ZZ$,
	\[
		\Hom_{\D(A)}(X,\Sigma^iY)\in\fl H^0(R).
	\]
\end{lem}
\begin{proof}
  Let $X\in\D_{\mathrm{fg}}(A)^{\le0}$.
	Choose finitely many representatives of generators of $H^0(X)$ to obtain a morphism $A^{\oplus r_0}\to X$ surjective on $H^0$. Its cone belongs to $\D_{\mathrm{fg}}(A)^{\le-1}$. Inductively, we have \cref{eq:fg-approx}. 
\end{proof}

The following result of Keller and Nicol\'as will be used for positive dg algebras. To avoid confusion with the standard $t$-structure truncations $\sigma^{\le n},\sigma^{\ge n}$, we denote the truncations for the canonical weight structure by $w_{\le n}$ and $w_{>n}$.

\begin{thm}\cite[Corollary~5.1]{KN}\label{thm:KN-weight}
	Let $B$ be a positive dg algebra over a commutative ring. Then $\D(B)$ admits a canonical weight structure whose two halves are
	\[
		\D(B)_{w>k}=\{X\in\D(B)\mid H^{\le k}(X)=0\},\qquad
		\D(B)_{w\le k}=\{X\in\D(B)\mid H^{>k}(X)=0\}.
	\]
	In particular, for every $X\in\D(B)$ and every $k\in\ZZ$ there is a weight decomposition
	\begin{equation}\label{eq:weight-decomposition}
		w_{>k}X\longrightarrow X\longrightarrow w_{\le k}X\longrightarrow
	\end{equation}
	such that the first morphism induces an isomorphism on $H^i$ for $i>k$ and the second induces an isomorphism on $H^i$ for $i\le k$. 
\end{thm}

\begin{nota}
	For a positive dg algebra $B$, we write
	\[
		W_B:=w_{\le0}B
	\]
	for a choice of the weight truncation of the free module $B$ supplied by \cref{thm:KN-weight}. Thus $H^i(W_B)=0$ for $i\neq0$ and $H^0(W_B)\simeq H^0(B)$.
\end{nota}

\begin{prop}\cite[Proposition~5.6, Lemma~6.2]{KN}\label{prop:proj}
	Let $B$ be a positive dg $R$-algebra that is locally of finite length. Then $W_B$ is a silting object in $\Dbfl(B)$. Moreover, if $X\in\Dbfl(B)$ has cohomology concentrated in degree zero, then $X\in\add W_B$.
\end{prop}

We finally recall the standard constructions which associate simple-minded collections to silting objects and conversely.

\begin{prop}\label{prop:silt-SMC}
	\begin{itemize}
		\item[(1)] Let $A$ be a locally finitely generated non-positive dg $R$-algebra. There is a map
		\[
			\phi:\silt(\per(A))\to\SMC(\Dbfl(A))
		\]
		characterized by
		\[
			\RHom_A(M,\phi(M))\in\D(R)^0.
		\]
		\item[(2)] Let $B$ be a positive dg $R$-algebra that is locally of finite length. There is a map
		\[
			\psi:\SMC(\per(B))\to\silt(\Dbfl(B))
		\]
		characterized by
		\[
			\RHom_B(\L,\psi(\L))\in\D(R)^0.
		\]
	\end{itemize}
\end{prop}
\begin{proof}
	$(1)$: Let $M$ be a silting object in $\per(A)$ and put $E:=\REnd_A(M)$. Since $M$ is silting, $E$ is non-positive. Derived Morita theory \cite[Lemma~6.1]{Ke} gives an equivalence
	\[
		F_M:=\RHom_A(M,-):\D(A)\xrightarrow{\sim}\D(E)
	\]
	sending $M$ to $E$.

	Let $\S_E$ be the standard simple-minded collection in $\Dbfl(E)$ and define
	\[
		\phi(M):=F_M^{-1}(\S_E).
	\]
	Since $F_M(M)=E$, we have $\RHom_A(M,\phi(M))\in\D(R)^0$. Conversely, if $\L\in\SMC(\Dbfl(A))$ satisfies $\RHom_A(M,\L)\in\D(R)^0$, then $F_M(\L)$ lies in the standard heart of $\Dbfl(E)$ and consists of finite-length modules, hence $F_M(\L)\subseteq\Filt\S_E$. It follows that $F_M(\L)\simeq\S_E$. This proves the characterization and uniqueness.

	$(2)$: The proof is analogous to that of $(1)$.
\end{proof}

%%%%%%%%%%%%%%%%%%%%%%%%%%%%%%%%%%%%%%%%%%%%%%%%%%%%%%%%%%%%%%%%%%%%%%%%
\section{Silting--simple-minded collection correspondence}\label{section:ST}

Fix a complete local noetherian non-positive commutative dg ring $R$ and a locally finitely generated non-positive dg $R$-algebra $A$. We prove \cref{main:ST} by Koszul duality.

Put
\[
	B:=\REnd_A(S_A),
	\qquad
	\Phi:=\RHom_A(-,S_A):\D(A)\longrightarrow\D(B^o)^o.
\]

\begin{lem}\label{lem:B-positive}
	The dg algebra $B$ is positive and locally of finite length. Moreover,
	\[
		\Phi(S_A)={}_BB,
		\qquad
		\Phi(A)\simeq W_{B^o}.
	\]
\end{lem}
\begin{proof}
	Since
	\[
		H^i(B)\simeq\Hom_{\D(A)}(S_A,\Sigma^iS_A),
	\]
	the standard $t$-structure gives $H^{<0}(B)=0$, while \cref{lem:Hom-length} gives finite length of every $H^i(B)$ over $H^0(R)$. We have 
	\[
		H^0(B)=\End_{H^0(A)}(\top H^0(A))\simeq\top H^0(A),
	\]
	which is semisimple. Thus $B$ is positive and locally of finite length.

	The identity $\Phi(S_A)={}_BB$ is tautological. Applying $\Phi$ to $A\to S_A\to C\to$ with $C\in\D(A)^{<0}$, we obtain
	\[
		\Phi(C)\to\Phi(S_A)\to\Phi(A)\to.
	\]
	We have $\Phi(C)\in\D(B^o)_{w>0}$ and $\Phi(A)\in\D(B^o)_{w\le0}$. Since the morphism $\Hom_{\D(A)}(S_A,S_A)\to\Hom_{\D(A)}(A,S_A)$ is an isomorphism, we have $\Phi(A)\simeq W_{B^o}$.
\end{proof}

We next isolate the completion argument. Here we use the models of $R$ and $A$ concentrated in non-positive degrees. Choose generators $\bar x_1,\ldots,\bar x_r$ of the maximal ideal $\m\subset H^0(R)$ and lifts $x_1,\ldots,x_r\in R^0$. Put $I:=(x_1,\ldots,x_r)\subset R^0$. For $n\ge1$, let
\[
	K_n:=K(R^0;x_1^n,\ldots,x_r^n)
\]
be the Koszul complex, placed in cohomological degrees $[-r,0]$, and put
\[
	A_n:=A\ten_{R^0}K_n.
\]
We use the standard transition morphisms $K_{n+1}\to K_n$: for one generator $x_i$, the map $K(R^0;x_i^{n+1})\to K(R^0;x_i^n)$ is the identity in degree zero and multiplication by $x_i$ in degree $-1$, and the general map is their tensor product.

\begin{lem}\label{lem:koszul-completion}
	The following assertions hold.
	\begin{itemize}
		\item[(1)] The canonical morphism
		\[
			A\longrightarrow\holim_n A_n
		\]
		is an isomorphism in $\D(A)$.
		\item[(2)] Each $A_n$ is locally of finite length. In particular,
		\[
			\sigma^{\ge-m}A_n\in\Dbfl(A)
		\]
		for every $m\ge0$.
		\item[(3)] The transition maps induce an isomorphism
		\[
			\hocolim_n\Phi(A_n)\xrightarrow{\sim}\Phi(A).
		\]
	\end{itemize}
\end{lem}
\begin{proof}
	$(1)$: The image of $I=(x_1,\ldots,x_r)$ in $H^0(R)$ is exactly $\m$. For every $j$ and $n\ge1$ we therefore have
	\[
		I^nH^j(A)=\m^nH^j(A).
	\]
	Since $H^j(A)$ is finitely generated over the complete noetherian local ring $H^0(R)$, the canonical map
	\[
		H^j(A)\longrightarrow\varprojlim_n H^j(A)/\m^nH^j(A)
	\]
	is an isomorphism. Hence $H^j(A)$ is $I$-adically complete as an $R^0$-module. By \cite[\href{https://stacks.math.columbia.edu/tag/091T}{Tag 091T}]{stacks-project}, it is derived $I$-complete. Since this holds for every $j$, \cite[\href{https://stacks.math.columbia.edu/tag/091U}{Tag 091U}]{stacks-project} implies directly that the possibly unbounded-below complex $A$ is derived $I$-complete in $\D(R^0)$. Applying the Koszul characterization \cite[\href{https://stacks.math.columbia.edu/tag/091Z}{Tag 091Z}]{stacks-project} gives
	\[
		A\xrightarrow{\sim}\holim_n\bigl(A\ten_{R^0}K_n\bigr)=\holim_nA_n
	\]
	in $\D(R^0)$. The morphism is $A$-linear. Since the forgetful functor $\D(A)\to\D(R^0)$ is conservative, it is therefore an isomorphism in $\D(A)$ as well.

	$(2)$: 
The finite free filtration of $K_n$ gives a finite filtration of $A_n=A\ten_{R^0}K_n$ whose subquotients are finite direct sums of shifts of $A$. Hence each $H^i(A_n)$ is finitely generated over $H^0(R)$. By the standard null-homotopy for Koszul complexes \cite[\href{https://stacks.math.columbia.edu/tag/0663}{Tag 0663}]{stacks-project}, the ideal
\[
	(\bar x_1^n,\ldots,\bar x_r^n)\subset H^0(R)
\]
annihilates $H^*(A_n)$. Its radical is $\m$, and therefore it is $\m$-primary. Thus each $H^i(A_n)$ has finite length over $H^0(R)$. Since both $A$ and $K_n$ are concentrated in non-positive degrees, so is $A_n$. Consequently $\sigma^{\ge-m}A_n$ has nonzero cohomology in only finitely many degrees, all of finite length, and hence belongs to $\Dbfl(A)$.

	$(3)$: Since $K_n$ is a finite free $R^0$-complex, tensor--Hom adjunction gives an isomorphism of right $B^o$-modules
	\[
		\Phi(A_n)=\RHom_A(A\ten_{R^0}K_n,S_A)\simeq\HHom_{R^0}(K_n,S_A).
	\]
	The elements $x_1,\ldots,x_r$ act trivially on $S_A$. For $0<j\le r$, the transition maps on $H^j\Phi(A_n)$ are multiplication by products of the $x_i$, hence zero. It follows that, as $R^0$-modules,
	\[
		\colim_n H^j\Phi(A_n)=
		\begin{cases}
			S_A & j=0,\\
			0 & j\ne0.
		\end{cases}
	\]
	Filtered colimits are exact, so cohomology commutes with the homotopy colimit and the canonical morphism gives
	\[
		\hocolim_n\Phi(A_n)\xrightarrow{\sim}S_A=\Phi(A).
	\]
\end{proof}

The next lemma collects the formal properties of the Koszul-dual functor which we need.

\begin{lem}\label{lem:routine}
	The functor $\Phi$ has the following properties.
	\begin{itemize}
		\item[(1)] For every $X\in\D(A)$ and $Y\in\Dbfl(A)$, the natural morphism
		\[
			\Phi_{X,Y}:\RHom_A(X,Y)\longrightarrow\RHom_{B^o}(\Phi(Y),\Phi(X))
		\]
		is an isomorphism.
		\item[(2)] For every $X\in\D(A)$, the canonical morphism
		\[
			\hocolim_m\Phi(\sigma^{\ge-m}X)\longrightarrow\Phi(X)
		\]
		is an isomorphism.
	\end{itemize}
\end{lem}
\begin{proof}
	$(1)$: For every $X\in\D(A)$, we have an isomorphism
	\[
		\RHom_A(X,S_A)=\Phi(X)\xrightarrow{\sim}\RHom_{B^o}(B,\Phi(X)).
	\]
	The claim follows from d\'evissage in the second variable.

	$(2)$: Applying $\Phi$ to the truncation triangle
	\[
		\sigma^{<-m}X\to X\to\sigma^{\ge-m}X\to
	\]
	shows that it is enough to prove $H^i\Phi(\sigma^{<-m}X)=0$ for $m\gg0$, for each fixed $i$. But
	\[
		H^i\Phi(\sigma^{<-m}X)=\Hom_{\D(A)}(\sigma^{<-m}X,\Sigma^iS_A),
	\]
	which vanishes for $m\gg0$ by the standard $t$-structure. Exactness of filtered colimits gives the claim.
\end{proof}

We now prove the double-centralizer statement needed for full faithfulness.

\begin{prop}\label{prop:double-centralizer}
	The natural biduality morphism
	\[
		\RHom_A(A,A)\longrightarrow\RHom_{B^o}(\Phi(A),\Phi(A))
	\]
	is a quasi-isomorphism of dg $R$-modules. Consequently, $\Phi$ restricts to an equivalence
	\[
		\per(A)\xrightarrow{\sim}\Dbfl(B^o)^o.
	\]
\end{prop}
\begin{proof}
	Consider the commutative diagram
	\begin{center}
	\begin{adjustbox}{max width=\textwidth}
		\begin{tikzcd}
			\RHom_A(A,A)\ar[r,"\Phi"]\ar[d,"\sim"sloped] & \RHom_{B^o}(\Phi(A),\Phi(A))\ar[d,"\sim"'sloped] \\
			\holim_n\RHom_A(A,A_n)\ar[r]\ar[d,"\sim"sloped] & \holim_n\RHom_{B^o}(\Phi(A_n),\Phi(A))\ar[d,"\sim"'sloped] \\
			\holim_n\holim_m\RHom_A(A,\sigma^{\ge-m}A_n)\ar[r,"\sim"] & \holim_n\holim_m\RHom_{B^o}(\Phi(\sigma^{\ge-m}A_n),\Phi(A)).
		\end{tikzcd}
	\end{adjustbox}
	\end{center}
	The upper-left vertical morphism is an isomorphism by \cref{lem:koszul-completion}(1). The upper-right vertical morphism is an isomorphism by \cref{lem:koszul-completion}(3), since $\RHom_{B^o}(-,\Phi(A))$ sends homotopy colimits in the first variable to homotopy limits. The lower-left vertical morphism is an isomorphism because the canonical map
	\[
		A_n\longrightarrow\holim_m\sigma^{\ge-m}A_n
	\]
	is an isomorphism: in each fixed cohomological degree the inverse system of cohomology groups is eventually constant, so the associated $\varprojlim^1$ term vanishes. The lower-right vertical morphism is an isomorphism by \cref{lem:routine}(2). Finally, the bottom horizontal morphism is an isomorphism by \cref{lem:routine}(1), because $\sigma^{\ge-m}A_n\in\Dbfl(A)$. Therefore the top horizontal morphism is an isomorphism.

For full faithfulness, the natural morphisms
\[
	\RHom_A(X,Y)\longrightarrow\RHom_{B^o}(\Phi(Y),\Phi(X))
\]
are isomorphisms for $X=Y=A$ by the top horizontal isomorphism. By d\'evissage in both variables, they are isomorphisms for all $X,Y\in\per(A)$. Thus $\Phi$ is fully faithful. Moreover, its essential image is $\thick\Phi(A)=\thick W_{B^o}=\Dbfl(B^o)$ by \cref{lem:B-positive,prop:proj}, proving the equivalence.
\end{proof}

\begin{prop}\label{prop:summary}
	The functor $\Phi:\D(A)\to\D(B^o)^o$ satisfies the following properties:
	\begin{itemize}
		\item it restricts to equivalences
		\[
			\Dbfl(A)\xrightarrow{\sim}\per(B^o)^o,
			\qquad
			\per(A)\xrightarrow{\sim}\Dbfl(B^o)^o;
		\]
		\item for every $X\in\per(A)$ and $Y\in\Dbfl(A)$, the morphism
		\[
			\Phi_{X,Y}:\RHom_A(X,Y)\longrightarrow\RHom_{B^o}(\Phi(Y),\Phi(X))
		\]
		is an isomorphism.
	\end{itemize}
\end{prop}
\begin{proof}
	The second equivalence and the Hom-isomorphism follows from \cref{prop:double-centralizer} and \cref{lem:routine}(1).

	By \cref{lem:routine}(1), the functor
	\[
		\Phi:\Dbfl(A)\longrightarrow\per(B^o)^o
	\]
	is fully faithful. Since $\Dbfl(A)=\thick S_A$ and $\Phi(S_A)=B$, its essential image is $\thick B=\per(B^o)$.
\end{proof}

We can now prove the silting--simple-minded correspondence. 

\begin{thm}\label{thm:ST}
	The following diagram is commutative:
	\begin{center}
		\begin{tikzcd}
			\silt(\per(A))\ar[r,leftrightarrow,"\Phi"]\ar[dd,"\phi","\text{\cref{prop:silt-SMC} }"'] & \silt(\Dbfl(B^o)) \\
			& & \\
			\SMC(\Dbfl(A))\ar[r,leftrightarrow,"\Phi"] & \SMC(\per(B^o))\ar[uu,"\psi","\text{ \cref{prop:silt-SMC}}"'].
		\end{tikzcd}
	\end{center}
	In particular,
	\[
		\phi:\silt(\per(A))\rightarrow\SMC(\Dbfl(A))
	\]
	is a bijection.
\end{thm}
\begin{proof}
	Let $M\in\silt(\per(A))$. By \cref{prop:summary},
	\[
		\RHom_{B^o}(\Phi\phi(M),\Phi(M))\in\D(R)^0.
	\]
	By the characterization in \cref{prop:silt-SMC}(2),
	\[
		\psi(\Phi\phi(M))=\Phi(M).
	\]
	Thus $\Phi=\psi\circ\Phi\circ\phi$ on $\silt(\per(A))$.

	For the converse identity, let $\L\in\SMC(\per(B^o))$ and put $N:=\psi(\L)\in\silt(\Dbfl(B^o))$. By the characterization of $\psi$,
	\[
		\RHom_{B^o}(\L,N)\in\D(R)^0.
	\]
	Applying the Hom-isomorphism in \cref{prop:summary} to $\Phi^{-1}(N)\in\per(A)$ and $\Phi^{-1}(\L)\in\Dbfl(A)$ gives
	\[
		\RHom_A\bigl(\Phi^{-1}(N),\Phi^{-1}(\L)\bigr)\in\D(R)^0.
	\]
	Hence the characterization of $\phi$ yields
	\[
		\phi\bigl(\Phi^{-1}(N)\bigr)=\Phi^{-1}(\L).
	\]
	Thus $\Phi^{-1}=\phi\circ\Phi^{-1}\circ\psi$ on $\SMC(\per(B^o))$, and the diagram commutes.
\end{proof}

%%%%%%%%%%%%%%%%%%%%%%%%%%%%%%%%%%%%%%%%%%%%%%%%%%%%%%%%%%%%%%%%%%%%%%%%%%%%%%%%%%%%%%%
\section{Derived base change}\label{section:bij}

In this section, let $R\to S$ be a local morphism of complete local noetherian non-positive commutative dg rings. We write $(H^0(R),\m,k)$ and $(H^0(S),\n,\ell)$ for the corresponding degree-zero local rings and residue fields. Unless otherwise stated, $A$ denotes a locally finitely generated non-positive dg $R$-algebra. 

\begin{rmk}\label{rmk:strict-model}
	We fix a strict dg $S$-algebra model for derived base change and the associated restriction-of-scalars functor. Choose a cofibrant resolution $Q\xrightarrow{\sim}A$ in associative dg $R$-algebras whose underlying dg $R$-module is K-flat (see \cite[Theorem~4.1]{SchwedeShipley00}). Then $Q\ten_R S$ is a dg $S$-algebra representing $A\ten_R^{\LL}S$, and the derived base-change morphism is represented by the zigzag of dg $R$-algebras
	\[
		A\xleftarrow{\sim}Q\longrightarrow Q\ten_R S.
	\]
	The right-hand arrow is the strict morphism $q\mapsto q\ten1$; in general, it does not give a strict morphism from $A$ to $Q\ten_R S$.

	Since $Q$ is quasi-isomorphic to $A$, it is non-positive, and so is $Q\ten_R S$. We use $Q\ten_R S$ as the strict model for $A\ten_R^{\LL}S$. Restriction of scalars $\iota\colon\D(Q\ten_R S)\to\D(A)$ means restriction along $Q\to Q\ten_R S$ followed by the equivalence $\D(Q)\simeq\D(A)$ induced by $Q\xrightarrow{\sim}A$. Its left adjoint is the corresponding derived extension of scalars, denoted by $-\ten_A^{\LL}(A\ten_R^{\LL}S)$.
\end{rmk}

The morphism $R\to S$ plays no role in \cref{subsection:SMC-fiber,subsection:bijections}, which concern only $R$ and its residue field $k$; it reappears in \cref{thm:R-S}.

\subsection{Silting objects}

We first record the formal behavior of silting objects under extension of scalars.

\begin{lem}\label{lem:basechange}
	Let $M\in\per(A)$ and $N\in\D(A)$. Then the morphism
  \[
		\RHom_A(M,N)\ten_R^{\LL}S
		\rightarrow
		\RHom_{A\ten_R^{\LL}S}(M\ten_R^{\LL}S,N\ten_R^{\LL}S)
	\]
  is an isomorphism.
\end{lem}
\begin{proof}
	It is an isomorphism if $M=A$, and the general case follows by d\'evissage in the first variable.
\end{proof}

\begin{prop}\label{prop:silt}
	If $M$ is a silting object in $\per(A)$, then $M\ten_R^{\LL}S$ is a silting object in $\per(A\ten_R^{\LL}S)$.
\end{prop}
\begin{proof}
	By \cref{lem:basechange}, we have
	\[
		\REnd_{A\ten_R^{\LL}S}(M\ten_R^{\LL}S)\simeq\REnd_A(M)\ten^{\LL}_RS.
	\]
  Since $\REnd_A(M)\in\D(R)^{\le0}$, we have $\REnd_{A\ten_R^{\LL}S}(M\ten_R^{\LL}S)\in\D(S)^{\le0}$. It remains to show generation.
	Since $A\in\thick M$, extension of scalars gives
	\[
		A\ten_R^{\LL}S\in\thick(M\ten_R^{\LL}S).
	\]
	Hence $M\ten_R^{\LL}S$ generates $\per(A\ten_R^{\LL}S)$.
\end{proof}

%%%%%%%%%%%%%
\subsection{Simple-minded collections and the closed fiber}\label{subsection:SMC-fiber}

Put
\[
	A_k:=A\ten_R^{\LL}k
\]
and let
\[
	\iota:\D(A_k)\longrightarrow\D(A)
\]
be restriction of scalars. We show that restriction preserves simple-minded collections. 

We work with a relative semifree resolution in order to retain the bimodule structure of the derived closed fiber. Since the augmentation $R\to k$ is surjective, we may choose a semifree dg $R$-algebra resolution
\[
	F=(T_R(V),d_F)\xrightarrow{\sim}k
\]
with a graded free $R$-module $V$ such that $V=V^{<0}$.

The underlying dg $R$-module of $F$ is semifree, hence K-flat over $R$. Put $B:=A\ten_R F$, a strict dg algebra model of $A_k$, and use the same symbol $\iota:\D(B)\to\D(A)$ for restriction of scalars. The relative diagonal has the exact sequence of dg $F$-bimodules
\begin{equation}\label{eq:relative-diagonal}
	0\longrightarrow\Omega\longrightarrow F\ten_R F\xrightarrow{\mu}F\longrightarrow0,
	\qquad \Omega:=\ker\mu.
\end{equation}
Let $\delta:F\to\Omega$ be the universal $R$-derivation. Since $F$ is the tensor algebra on $V$, the assignment $f_1\ten v\ten f_2\mapsto f_1\delta(v)f_2$ gives an isomorphism
\[
	F\ten_R V\ten_R F\xrightarrow{\sim}\Omega
\]
of graded $F$-bimodules. 

For $n\ge0$, let $V_{\ge-n}$ be the graded free $R$-submodule of $V$ generated by the chosen generators of degree at least $-n$, and let $\Omega_n\subseteq\Omega$ correspond, under the graded identification above, to $F\ten_R V_{\ge-n}\ten_R F$. Thus $\Omega_0=0$ and $\sigma^{<-n}(\Omega/\Omega_n)=\Omega/\Omega_n$. By construction, $\Omega_n$ is a dg $F$-sub-bimodule satisfying
\begin{equation}\label{eq:relative-diagonal-filtration}
	\Omega_n/\Omega_{n-1}
	\simeq\bigl(\Sigma^n(F\ten_R F)\bigr)^{(I_n)}
\end{equation}
as dg $F$-bimodules, where $I_n$ denotes the set of generators of degree $-n$. 

\begin{lem}\label{lem:relative-filtration}
	Let $X\in\D^-(B)$, and choose $u\in\ZZ$ such that $X\in\D(B)^{\le u}$. Set
	\[
		P_X:=\iota X\ten_A^{\LL}B,
		\qquad K_X:=\cocone(P_X\longrightarrow X),
	\]
	where the arrow is the adjunction counit. Then $P_X\in\D(B)^{\le u}$, and after choosing a model, $K_X$ has a filtration by dg right $B$-submodules $K_{X,0}=0\subseteq K_{X,1}\subseteq\cdots$ such that
	\[
		K_{X,n}/K_{X,n-1}\simeq(\Sigma^nP_X)^{(I_n)},
		\qquad K_X/K_{X,n}\in\D(B)^{\le u-n-1}.
	\]
\end{lem}
\begin{proof}
	Since $F$ is K-flat over $R$, the left $A$-module $B=A\ten_R F$ is K-flat. Hence $P_X$ is represented by $X\ten_A B\simeq X\ten_R F$. The counit is $x\ten f\mapsto x(1\ten f)$.

	The multiplication $\mu$ in \eqref{eq:relative-diagonal} has a dg left $F$-linear section $f\mapsto f\ten1$. Thus tensoring the sequence with $X$ over the left copy of $F$ is exact, and yields an exact sequence of dg right $B$-modules
	\[
		0\longrightarrow X\ten_F\Omega\longrightarrow
		X\ten_R F\longrightarrow X\longrightarrow0.
	\]
	In particular, $K_X$ is represented by $X\ten_F\Omega$. If we put $K_{X,n}:=X\ten_F\Omega_n$, then these satisfy the desired properties by \eqref{eq:relative-diagonal-filtration}.
\end{proof}

\begin{lem}\label{lem:restriction-Hom}
	Let $X\in\D^-(A_k)$ and $Y\in\D^+(A_k)$. Assume that $\RHom_{A_k}(X,Y)\in\D(k)^{\ge0}$.
	Then $\RHom_A(\iota X,\iota Y)\in\D(R)^{\ge0}$, and the morphism
	\[
		\Hom_{\D(A_k)}(X,Y)
		\longrightarrow
		\Hom_{\D(A)}(\iota X,\iota Y)
	\]
	is an isomorphism.
\end{lem}
\begin{proof}
	We work with the strict model $B$. Choose $u,v\in\ZZ$ such that $X\in\D(B)^{\le u}$ and $Y\in\D(B)^{\ge v}$. By tensor--Hom adjunction, we have 
	\[
		\Hom_{\D(A)}(\iota X,\Sigma^{-i}\iota Y)\simeq\Hom_{\D(B)}(P_X,\Sigma^{-i}Y).
	\]
	Since $P_X\in\D(B)^{\le u}$, we have $\Hom_{\D(B)}(P_X,\Sigma^{-i}Y)=0$ for $i\gg0$. It suffices to show that the morphism $\Hom_{\D(B)}(X,\Sigma^{-i}Y)\to\Hom_{\D(B)}(P_X,\Sigma^{-i}Y)$ is an isomorphism provided that $\Hom_{\D(B)}(P_X,\Sigma^{-j}Y)=0$ for every $j>i$.
    
	Suppose that $\Hom_{\D(B)}(P_X,\Sigma^{-j}Y)=0$ for every $j>i$. By \cref{lem:relative-filtration}, the successive quotients of $K_{X,n}$ are coproducts of positive shifts of $P_X$. Hence
	\[
		\Hom_{\D(B)}(\Sigma^{0,1}K_{X,n},\Sigma^{-i}Y)=0.
	\]
	For $n\ge u-v$, the bound $K_X/K_{X,n}\in\D(B)^{<v}$ gives 
	\[
		\Hom_{\D(B)}(\Sigma^{0,1}K_X/K_{X,n},\Sigma^{-i}Y)=0.
	\]
	Thus $\Hom_{\D(B)}(\Sigma^{0,1}K_X,\Sigma^{-i}Y)=0$, and the counit triangle yields the desired isomorphism
	\[
		\Hom_{\D(B)}(X,\Sigma^{-i}Y)\xrightarrow{\sim}\Hom_{\D(B)}(P_X,\Sigma^{-i}Y).
	\]
\end{proof}

\begin{prop}\label{prop:SMC}
	Let $\L\subseteq\Dbfl(A_k)$ be a simple-minded collection. Then $\iota(\L)$ is a simple-minded collection in $\Dbfl(A)$.
\end{prop}
\begin{proof}
	Since every object of $\L$ has finite total dimension over $k$, its restriction has finite total length over $H^0(R)$; hence $\iota(\L)\subseteq\Dbfl(A)$. By \cref{lem:restriction-Hom}, restriction preserves all negative Hom-vanishing between objects of $\L$ and induces isomorphisms on degree-zero Hom spaces. Hence $\iota(\L)$ is a semibrick and satisfies
	\[
		\Hom_{\D(A)}(\iota(\L),\Sigma^{<0}\iota(\L))=0.
	\]

	It remains to prove generation. Since the degree-zero cohomology of a derived tensor product of non-positive dg algebras is the tensor product of their degree-zero cohomologies,
	\[
		H^0(A_k)\simeq H^0(A)\ten_{H^0(R)}k.
	\]
	It follows that $\top H^0(A_k)=\top H^0(A)$, and so
	\[
		\S_A\simeq\iota(\S_{A_k}).
	\]
	Since $\S_{A_k}\subseteq\thick\L$, we obtain
	\[
		\S_A\subseteq\thick\iota(\L).
	\]
	Because $\thick\S_A=\Dbfl(A)$, the collection $\iota(\L)$ generates $\Dbfl(A)$.
\end{proof}

\subsection{Bijections}\label{subsection:bijections}

We now prove the closed-fiber theorem.

\begin{thm}\label{thm:bij-fib}
	The following diagram is commutative:
	\begin{center}
		\begin{tikzcd}
			\silt(\per(A))\ar[rr,"-\ten_R^{\LL}k"']\ar[dd,leftrightarrow,"\phi"'] & & \silt(\per(A_k))\ar[dd,leftrightarrow,"\phi"] \\
			& & \\
			\SMC(\Dbfl(A)) & & \SMC(\Dbfl(A_k))\ar[ll,"\iota"'].
		\end{tikzcd}
	\end{center}
	In particular,
	\[
		-\ten_R^{\LL}k:\silt(\per(A))\rightarrow\silt(\per(A_k))
	\]
	is a poset isomorphism. 
\end{thm}
\begin{proof}
	Let $M\in\silt(\per(A))$. Tensor--Hom adjunction gives
	\[
		\RHom_A\bigl(M,\iota(\phi(M\ten_R^{\LL}k))\bigr)
		\simeq
		\RHom_{A_k}(M\ten_R^{\LL}k,\phi(M\ten_R^{\LL}k))
		\in\D(k)^0.
	\]
	By the characterization in \cref{prop:silt-SMC}(1), we have
	\[
		\phi(M)=\iota(\phi(M\ten_R^{\LL}k)).
	\]
	Thus the diagram commutes.

	Let $\L\in\SMC(\Dbfl(A_k))$ and put $M:=\phi_A^{-1}(\iota(\L))$, where subscripts distinguish the two silting--simple-minded correspondences. By adjunction,
	\[
		\RHom_{A_k}(M\ten_R^{\LL}k,\L)\simeq\RHom_A(M,\iota(\L))\in\D(R)^0.
	\]
	The characterization in \cref{prop:silt-SMC}(1) gives $\phi_{A_k}(M\ten_R^{\LL}k)=\L$. Thus
	\[
		(-\ten_R^{\LL}k)\circ\phi_A^{-1}\circ\iota=\phi_{A_k}^{-1}.
	\]
	This proves the bijectivity of $-\ten^\LL_Rk\colon\silt(\per(A))\to\silt(\per(A_k))$. Order preservation and reflection follow from 
  \[
    \RHom_A(M,N)\in\D(R)^{\le0}\Leftrightarrow\RHom_{A_k}(M\ten^\LL_Rk,N\ten^\LL_Rk)\in\D(k)^{\le0}\qquad (M,N\in\per(A)).
  \]
  Indeed, \cref{lem:basechange} gives 
  \[
    \RHom_{A_k}(M\ten^\LL_Rk,N\ten^\LL_Rk)\simeq\RHom_A(M,N)\ten^\LL_Rk,
  \]
  and Nakayama's lemma can be applied.
\end{proof}

We now return to the local morphism $R\to S$ and assume that it induces an isomorphism $k\xrightarrow{\sim}\ell$ on residue fields. We identify both residue fields with $k$.

\begin{thm}\label{thm:R-S}
	The assignment
	\[
		M\longmapsto M\ten_R^{\LL}S
	\]
	induces a poset isomorphism
	\[
		\silt(\per(A))\xrightarrow{\sim}\silt(\per(A\ten_R^{\LL}S)).
	\]
\end{thm}
\begin{proof}
	By associativity of derived tensor product,
	\[
		(A\ten_R^{\LL}S)\ten_S^{\LL}k\simeq A\ten_R^{\LL}k.
	\]
	We therefore have a commutative diagram
	\begin{center}
		\begin{tikzcd}
			\silt(\per(A))\ar[r,"-\ten_R^{\LL}S"]\ar[d,"-\ten_R^{\LL}k"',"\sim"sloped] & \silt(\per(A\ten_R^{\LL}S))\ar[d,"-\ten_S^{\LL}k","\sim"'sloped] \\
			\silt(\per(A\ten_R^{\LL}k))\ar[r,equal] & \silt(\per((A\ten_R^{\LL}S)\ten_S^{\LL}k)).
		\end{tikzcd}
	\end{center}
	The vertical maps are poset isomorphisms by \cref{thm:bij-fib}. Hence the upper horizontal map is a poset isomorphism.
\end{proof}

\subsection{Examples of derived base change}\label{subsection:examples}

\begin{eg}\label{eg:triangular-family}
	Let $k$ be a field, let $R=k[\![x]\!]$ with $|x|=0$, and for $n\ge1$ set
	\[
		\Lambda_n=\begin{pmatrix}
			R&0\\
			R/(x^n)&R
		\end{pmatrix}.
	\]
	This is a module-finite, non-flat $R$-algebra. The free resolution
	\[
		0\longrightarrow R\xrightarrow{x^n}R\longrightarrow R/(x^n)\longrightarrow0
	\]
	yields a dg $R$-algebra model with two idempotents and two parallel arrows $\alpha,\beta$, specified by
	\[
		|\alpha|=0,\qquad |\beta|=-1,\qquad d(\beta)=x^n\alpha.
	\]
	There are no paths of length two. The model is degreewise $R$-free, and hence the derived closed fiber is represented by the graded Kronecker algebra
	\[
		B=k\left(\begin{tikzcd}[column sep=4em,ampersand replacement=\&]
			1 \ar[r,bend left=18,"\alpha"]\ar[r,bend right=18,"\beta"'] \& 2
		\end{tikzcd}\right),\qquad |\alpha|=0,\quad |\beta|=-1,\quad d_B=0.
	\]
	In particular, $\Lambda_n\ten_R^{\LL}k\simeq B$ for every $n$, whereas the ordinary fiber $\Lambda_n\ten_Rk$ contains only the degree-zero arrow. Thus the degree $-1$ arrow records the non-flatness of $\Lambda_n$ over $R$. By \cref{thm:bij-fib},
	\[
		\silt(\per(\Lambda_n))\xrightarrow{\sim}\silt(\per(B)),
	\]
	so these silting posets are independent of $n$. The graded Kronecker algebra $B$ is silting-discrete in the sense of \cite{A13} by \cite[Corollary~2.18(1)]{AH24}; hence so is every $\Lambda_n$.
\end{eg}

\begin{eg}\label{eg:ordinary-vs-derived}
	Let $k$ be a field, let $R=k[\![x]\!]$, and consider the quiver and the module-finite $R$-algebra
	\[
		Q:\quad 2\xleftarrow{\alpha}1\xrightarrow{\beta}3,
		\qquad \Lambda=RQ/(x\alpha,x\beta).
	\]
	The ordinary closed fiber is $\Lambda\ten_R k\simeq kQ$, a path algebra of Dynkin type $A_3$, and is therefore silting-discrete.

	To compute the derived closed fiber, add two arrows $\alpha',\beta'$ of degree $-1$ to $Q$, obtaining the graded quiver
	\[
\widetilde Q:\quad
\begin{tikzcd}[ampersand replacement=\&, column sep=4.5em]
2
\& 1 \ar[l,bend left=18,"\alpha"] \ar[l,bend right=18,swap,"\alpha'"]
      \ar[r,bend left=18,"\beta"] \ar[r,bend right=18,swap,"\beta'"]
\& 3
\end{tikzcd}
\qquad
|\alpha|=|\beta|=0,\quad |\alpha'|=|\beta'|=-1.
\]
	Equip $F=R\widetilde Q$ with the differential
	\[
		d(\alpha')=x\alpha,\qquad d(\beta')=x\beta,
		\qquad d(\alpha)=d(\beta)=0.
	\]
	The natural dg $R$-algebra morphism $F\to\Lambda$ is a quasi-isomorphism. Since $F$ is degreewise free over $R$, it computes the derived fiber:
	\[
		\Lambda\ten_R^{\LL}k\simeq F\ten_R k=(k\widetilde Q,0),
	\]
	By the classification in \cite[Theorem~3.5]{FGraded26}, $k\widetilde Q$ is silting-indiscrete. By \cref{thm:bij-fib}, the same holds for $\Lambda$.

\begin{eg}\label{eg:dg-gentle}
	Let $k$ be a field and let
	\[
		R=k[t],\qquad |t|=-2,\qquad d(t)=0.
	\]
	Thus $R$ is a complete local noetherian non-positive commutative dg ring in the sense of \cref{dfn:base}, with $H^0(R)=k$. Let $Q$ be the graded $2$-cycle
	\[
		\begin{tikzcd}[column sep=4em]
			1 \ar[r,bend left=18,"a"] & 2 \ar[l,bend left=18,"b"]
		\end{tikzcd}
	\]
	with $|a|=0$ and $|b|=-2$, and let $A=kQ$ with trivial differential. Put
	\[
		c:=ab+ba.
	\]
	Then $c$ is a central homogeneous element of degree $-2$. We make $A$ a dg $R$-algebra by the morphism
	\[
		R\longrightarrow A,\qquad t\longmapsto c.
	\]
	As a graded $R$-module it is free of rank four:
	\[
		A\simeq Re_1\oplus Re_2\oplus Ra\oplus Rb.
	\]
	It follows that the derived closed fiber agrees with the ordinary tensor product:
	\[
		A\ten_R^{\LL}k
		\simeq A\ten_R k
		\simeq kQ/(ab,ba).
	\]
	The closed fiber is therefore the radical-square-zero graded $2$-cycle
	\[
		\begin{tikzcd}[column sep=4em]
			1 \ar[r,bend left=18,"a"] & 2 \ar[l,bend left=18,"b"]
		\end{tikzcd}
		\qquad (ab=ba=0),
	\]
	with $|a|=0$ and $|b|=-2$. It is a finite-dimensional graded gentle one-cycle algebra. Following Kalck--Yang \cite[Section~9]{KalckYang18}, write $d^\pm$ for the number of relations oriented in the $\pm$-direction around the cycle, minus the sum of the degrees of the arrows oriented in that direction. The graded clock condition is $d^+=d^-$. With the orientation $a,b$ clockwise, both relations are clockwise and $|a|+|b|=-2$, so $d^+=2-(-2)=4$ and $d^-=0$. Thus the graded clock condition fails. Hence it is derived-discrete and silting-discrete by \cite[Theorem~1.3]{FX26}. By \cref{thm:bij-fib}, extension of scalars induces a poset isomorphism
	\[
		\silt(\per(A))
		\xrightarrow{\sim}
		\silt\bigl(\per(kQ/(ab,ba))\bigr),
	\]
	and in particular $A$ is silting-discrete.
\end{eg}

\begin{rmk}\label{rmk:different-residue}
	The common-residue-field assumption in \cref{thm:R-S} cannot be omitted in general. Already for ordinary rings, take $R=\FF_2$, $A=\FF_4$, and $S=\FF_4$. Then
	\[
		A\ten_R^{\LL}S=\FF_4\ten_{\FF_2}\FF_4\simeq\FF_4\times\FF_4,
	\]
	and the map $\silt(\FF_4)\to\silt(\FF_4\times\FF_4)$ is not surjective.
\end{rmk}

%%%%%%%%%%%%%%%%%%%%%%%%%%%%%%%%%%%%%%%%%%%%%%%%%%%%%%%%%%%%%%%%%%%%%%%%%%%%%%%%%%%%%%%%
\section{Comparison with previous results}\label{section:relation}

In this section we specialize the preceding results to ordinary complete local noetherian rings and compare them with known reduction and base-change theorems. Thus, throughout this section, $(R,\m,k)$ is an ordinary commutative complete local noetherian ring and $\Lambda$ is a module-finite $R$-algebra.

Kimura \cite{K} proved a reduction theorem for two-term silting complexes modulo an arbitrary two-sided ideal $I\subseteq\m\Lambda$. We recover this general ideal case.

\begin{cor}\label{cor:Kimura-max}\cite[Theorem~1.5]{K}
	Let $I\subseteq\m\Lambda$ be a two-sided ideal. Then derived extension of scalars
	\[
		-\ten_\Lambda^{\LL}\Lambda/I:\twosilt(\Lambda)\xrightarrow{\sim}\twosilt(\Lambda/I)
	\]
	is an isomorphism of posets.
\end{cor}
\begin{proof}
	Put $\Gamma:=\Lambda/I$. By \cref{thm:bij-fib} and \cite[Proposition~A.3]{BY13}, two-term silting reduction gives poset isomorphisms
	\[
		\twosilt(\Lambda)\xrightarrow{\sim}\twosilt(\Lambda/\m\Lambda),
		\qquad
		\twosilt(\Gamma)\xrightarrow{\sim}\twosilt(\Gamma/\m\Gamma).
	\]
	Because $I\subseteq\m\Lambda$, the two degree-zero closed fibers coincide: $\Gamma/\m\Gamma\simeq\Lambda/\m\Lambda$. 
\end{proof}

We next compare with the Tor-independent base-change results of Gnedin \cite{G}. Other aspects of the behavior of silting objects under base change, including localization and quotient by a regular central element, were studied by Zhu, Wei, and Di \cite{ZWD25}. Let $(S,\n,k)$ be an ordinary commutative complete local noetherian $R$-algebra with the same residue field. If
\[
	\Tor^R_{>0}(\Lambda,S)=0,
\]
then the canonical morphism
\[
	\Lambda\ten_R^{\LL}S\longrightarrow\Lambda\ten_RS
\]
is a quasi-isomorphism. Hence \cref{thm:R-S} gives the following statement.

\begin{cor}\label{cor:Gnedin}
	If $\Tor^R_{>0}(\Lambda,S)=0$, then derived extension of scalars, represented on perfect complexes by termwise tensor product with $S$, induces a poset isomorphism
	\[
		\silt(\per(\Lambda))\xrightarrow{\sim}\silt(\per(\Lambda\ten_RS)).
	\]
\end{cor}

For quotient maps $S=R/\mathfrak a$ satisfying the Tor-independence hypothesis, this gives a base-change bijection of the same type as those studied in \cite[Theorem~A, Theorem~6.5]{G}. The hypotheses are not identical: Gnedin treats more general quotient ideals under additional lifting hypotheses, whereas \cref{cor:Gnedin} is the direct specialization of our derived base-change theorem when the derived tensor product is underived.

\end{eg}

\bibliographystyle{alpha}
\bibliography{my}

\end{document}